\documentclass[dvipsnames]{article}
\pdfoutput=1 
\usepackage[utf8]{inputenc}
\usepackage[T1]{fontenc}
\usepackage{graphicx}

\usepackage{amsmath,amsfonts,amsthm,amssymb,mathtools}
\usepackage{authblk}
\usepackage[colorlinks]{hyperref}
\hypersetup{
    linkcolor={red!50!black},
    citecolor={blue!50!black},
    urlcolor={blue!80!black}
}
\usepackage{tikz}
\usepackage{csquotes}
\usepackage{cleveref}
\usepackage{authblk}

\newtheorem{thm}{Theorem}[section]
\newtheorem{lem}[thm]{Lemma}
\newtheorem{defn}[thm]{Definition}
\newtheorem{prop}[thm]{Proposition}
\newtheorem{exam}[thm]{Example}

\theoremstyle{remark}
\newtheorem{rem}{Remark}

\newcommand{\N}{\mathbb{N}}
\newcommand{\R}{\mathbb{R}}

\newcommand{\C}{\mathbb{C}}
\newcommand{\K}{\mathbb{K}}
\newcommand{\calM}{\mathcal{M}}
\newcommand{\calS}{\mathcal{S}}

\newcommand{\setsep}{\;;\;}%{\;\mid\;}

\DeclareMathOperator{\trace}{trace} % diagonalization
\DeclareMathOperator{\rank}{rank}

\newcommand{\norm}[1]{\left\lVert#1\right\rVert} 
\newcommand{\abs}[1]{\left\lvert#1\right\rvert} 
\newcommand{\gl}{\mathbf{GL}}
\newcommand{\vect}[1]{\begin{pmatrix}#1\end{pmatrix}}
 
\newcommand{\nnorm}[1]
{{\left\vert\kern-0.25ex\left\vert\kern-0.25ex\left\vert #1 
    \right\vert\kern-0.25ex\right\vert\kern-0.25ex\right\vert}}

\author[1]{Jeremias Epperlein\thanks{\href{mailto:jeremias.epperlein@uni-passau.de}{jeremias.epperlein@uni-passau.de}}}
\author[1]{Fabian Wirth\thanks{\href{mailto:fabian.lastname@uni-passau.de}{fabian.lastname@uni-passau.de}}}

\affil[1]{University of Passau, %
    Faculty of Computer Science and Mathematics, %
    Innstraße 33, %
    94032 Passau, Germany}

\title{Auerbach bases, projection constants, and the joint spectral radius of
principal submatrices}

\begin{document}
\maketitle

\begin{abstract}
It is shown that compact sets of complex matrices can always be brought, via similarity transformation, into a form where all matrix entries are bounded in absolute value by the joint spectral radius (JSR). The key tool for this is that every extremal norm of a matrix set admits an Auerbach basis; any such basis gives rise to a desired coordinate system.
An immediate implication is that all diagonal entries — equivalently, all one-dimensional principal submatrices — are uniformly bounded above by the JSR.
It is shown that the corresponding bounding property does not hold for higher dimensional principal submatrices. 
More precisely, we construct finite matrix sets for which, across the entire similarity orbit, the JSRs of all higher-dimensional principal submatrices exceed that of the original set. This shows that the bounding result does not extend to submatrices of dimension greater than one.
The constructions rely on tools from the geometry of finite-dimensional Banach spaces, with projection constants of norms playing a key role.
Additional bounds of the JSR of principal submatrices are obtained using John's ellipsoidal approximation and known estimates for projection constants.
\end{abstract}

%\begin{keywords}
%Joint spectral radius, Auerbach basis, principal submatrix, projection %constants, Banach space geometry.
%\end{keywords}

%\msc{15A18, 47A30, 52A21, 15A42}
% MSC:
% 15A18 Eigenvalues, singular values, and eigenvectors 
% 47A30 Norms (inequalities, more than one norm, etc.) of linear operators
% 52A21 Convexity and finite-dimensional Banach spaces  
% 15A42 Inequalities involving eigenvalues and eigenvectors 

\section{Introduction}

The joint spectral radius (JSR) of a compact set of matrices describes the maximal exponential growth rate of matrix products formed from the set.
It was introduced by Rota and Strang in \cite{RotaStra60}, and a comprehensive and accessible introduction can be found in the monograph by Jungers \cite{Jungers}.
The JSR has numerous applications — e.g. in wavelet theory \cite{daubechies1992sets, lagariasFinitenessConjectureGeneralized1995}, and in dynamical systems, where it characterizes the exponential stability of discrete-time linear switched systems with arbitrary switching \cite{shorten2007stability}. Further applications are discussed in \cite{Jungers}.

Recently, there has been a renewed interest in the question of simultaneous similarity and studying similarity orbits of matrix sets in order to understand properties of the joint spectral radius. This has been carried out with great success for pairs of $2\times 2$ matrices in \cite{bochi2024spectrum,laskawiec2024partial}, where a complete set of invariants developed by Friedland in \cite{FRIEDLAND1983189} has been applied. 
In particular, \cite{bochi2024spectrum, laskawiec2024partial} examine spectrum maximizing products (SMPs). These are finite products of matrices that realize the JSR via their spectral radius. It is known that SMPs do not always exist, even for finite sets of matrices. After the seminal paper \cite{lagariasFinitenessConjectureGeneralized1995} examples of finite matrix sets without SMPs have been shown to exist in increasingly concrete form in \cite{blondel2003elementary,hareExplicitCounterexampleLagarias2011}. Despite this, many successful algorithms for computing the JSR rely on the assumption that SMPs exist and are even unique \cite{guglielmi2008algorithm, guglielmi2013exact}. It is a result of \cite{bochi2024spectrum} that there are open sets of matrix sets for which SMPs exist but are not unique.

These developments highlight the importance of understanding how structure and representation affect the behavior of the JSR — a theme that connects naturally to the broader normal form problem for compact matrix sets. Here one seeks canonical representatives of similarity orbits. Even for pairs of matrices this problem is notoriously hard.
Gelfand and Ponomarev demonstrated in \cite{gel1969remarks} that classifying pairs of commuting matrices under simultaneous similarity is as hard as the broader task of classifying $n$-tuples of matrices, which may not commute, under simultaneous similarity.
In representation theory, this problem is even used to characterize \enquote{wildness}, see \cite{BELITSKII2003203}.
More specifically, a classification problem is considered to be \enquote{wild} if its solution can be used to classify matrix tuples up to simultaneous conjugacy. For positive results on this classification problem, see \cite{FRIEDLAND1983189} for an algebraic geometric solution,
and \cite{chistovPolynomialTimeAlgorithms1997} for a solution in the sense of decidability.
For an algorithmic solution of the harder normal form problem, see
\cite{belitskiiNormalFormsMatrix2000}
and
\cite{sergeichukCanonicalMatricesLinear2000}.

In this paper, we study a simpler but more general question: given a compact set of square matrices, does its similarity orbit contain an instance in which all entries of all matrices are bounded in absolute value by the joint spectral radius? Since the JSR is invariant under simultaneous similarity, it is natural to ask whether this invariance is reflected by such a concrete coordinate-wise bound. We answer this question affirmatively for all matrix sets with positive JSR. 

Our first result, proved in Section~\ref{sec:auerbach}, establishes that if the JSR is strictly positive, then the similarity orbit of the matrix set contains an instance in which all entries of all matrices are bounded above in absolute value by the JSR. The idea of proof is to show that it is always possible to transform the matrix set in such a way that there is an extremal norm sandwiched between the $1$- and the $\infty$-norm.
The proof is based on the existence of an extremal norm whose associated Auerbach basis defines the desired coordinate system. These bases, classical objects in Banach space theory, consist of biorthogonal sequences with unit norm in both the norm and its dual. They were first shown to exist by Auerbach \cite{auerbach1930area} with independent derivations by Day, respectively Taylor \cite{day1947polygons, taylor1947geometric}. Recent work by Weber and Ziemke \cite{weber2017pelczynski} shows that for $d>2$, there exist at least $(d-1)d/2+1$ such bases (up to natural equivalences). It was already noticed in the original paper by \cite{RotaStra60} that norms are instrumental in the analysis of the joint spectral radius, and this has been reconfirmed, e.g. in \cite{Elsn95, Bara88, Wirt02, morris2010criteria, bochi2024spectrum}. In general, many of the numerical procedures available for the computation of the joint spectral radius can be interpreted in terms of extremal norms, see e.g. \cite{guglielmi2008algorithm, guglielmi2013exact}. 

A natural question is whether a similar result holds for principal submatrices, i.e. if, given a particular submatrix pattern, it is possible to find an element of the similarity orbit such that the JSR of the induced set of submatrices is bounded by the JSR of the original matrix set. This question was motivated by our recent work on regularity properties of the JSR. In \cite{epperlein2025joint}, we showed that the JSR is pointwise Hölder continuous, and even locally Hölder continuous in dimension $d=2$. For a brief time, we hoped that this two-dimensional regularity could be lifted inductively to higher dimensions by working with principal submatrices and exploiting favorable representatives in the similarity orbit. The result of the present paper show that such an approach fails in general.

In Section~\ref{sec:higher-dimensionalPS}, we show that there exist finite matrix sets with the following property: across the entire similarity orbit, and for every principal submatrix pattern of order between $2$ and $d-1$, the JSR of the corresponding set of submatrices strictly exceeds that of the full set. That is, none of these submatrix JSRs are ever brought below the global JSR by similarity transformation. In fact, the principal submatrix JSRs are uniformly bounded away from the global JSR.

This phenomenon is intimately linked to the geometry of the Banach space defined by an extremal norm. The relevant obstruction lies in the size of certain projection constants: the norms of linear projections from the space onto subspaces of dimension $2$ to $d-1$. If all such projection norms are uniformly bounded away from $1$, then the desired similarity transformation does not exist.

We refer to norms with this property as shady, as their unit balls cast a persistent “shadow” on intermediate-dimensional subspaces. While classical projection constants have been widely studied \cite{grunbaumProjectionConstants1960, chalmersProofGrunbaumConjecture2010, deregowskaSimpleProofGrunbaum2023, johnsonBasicConceptsGeometry2001}, the specific constant we require — the minimal norm of projections of intermediate rank — has received comparatively little attention. Notable exceptions include \cite{bosznayNormsProjections1986, kobosUniformEstimateRelative2018, kobosHYPERPLANESFINITEDIMENSIONALNORMED2015}. Our construction hinges on identifying shady norms, whose geometry precludes uniform control over the JSRs of principal submatrices across the similarity orbit. A computational approach to this problem is presented in \cite{epperlein2025shadiness}.

The paper is organized as follows. We briefly collect some well known facts in \Cref{sec:preliminaries}. In \Cref{sec:auerbach} we show how Auerbach bases may be used to construct a transformation that normalizes the entries of the matrices in a given set. \Cref{sec:higher-dimensionalPS} is devoted to higher-dimensional principal submatrices and it is shown that it is in general not possible to uniformly normalize the joint spectral radius associated to such submatrix sets. Two types of estimates for the joint spectral radius of the higher dimensional principal submatrices are provided. These rely either on approximation techniques using John ellipsoids or on the Grünbaum-Deregowska-Lewandowska theorem, \cite{grunbaumProjectionConstants1960,deregowskaSimpleProofGrunbaum2023}. The results relate to an approximation theorem due to \cite{ando1998simultaneous} where it was shown to which extent quadratic norms can be used to approximate the joint spectral radius. 
We briefly comment on pairs of matrices in \Cref{sec:pairs} and conclude in \Cref{sec:outlook}.

\section{Preliminaries}
\label{sec:preliminaries}

Let $\N$ be the set of natural numbers including $0$. The real and complex field are denoted by $\R$, $\C$ and $\R_{\geq0}:=[0,\infty)$. The Euclidean norm on $\R^d, \C^d$ is denoted by $\|\cdot\|_2$ and this also denotes the induced operator norm, i.e., the spectral norm on $\R^{d\times d},\C^{d\times d}$.

Let $\K=\R,\C$, $d\geq 1$. For a bounded, nonempty set of matrices $\mathcal{M}\subseteq\K^{d\times d}$ we consider the set of arbitrary products of length $t$ defined by 
\begin{equation*}
    {\calS}_t:= \{ A(t-1) \dots A(0) \setsep A(s) \in {\calM}, s=0,\ldots,t-1 \} \,.
\end{equation*}
The \emph{joint spectral radius} of $\mathcal{M}$ is defined as 
\begin{equation}
    \rho(\mathcal{M)} := \lim_{t\to \infty} \sup \{ \|S\|_2 \setsep S\in {\calS}_t \}^{1/t}.
\end{equation}
It is known that taking the closure of $\mathcal{M}$ does not change the value of the joint spectral radius. Thus we will assume that $\mathcal{M}$ is compact from now on.

A helpful formulation of the joint spectral radius is in terms of operator norms. Given a norm $\|\cdot\|$ on $\K^d$ and its induced operator norm also denoted by $\|\cdot\|$, we define
\begin{equation*}
    \|\mathcal{M}\| := \max \{\|A\| ; A \in \mathcal{M}\}.
\end{equation*}
Then it is known, \cite{RotaStra60}, \cite{Jungers}, that
\begin{equation}
    \rho(\mathcal{M)} = \inf \{ \|\mathcal{M}\| ; \|\cdot\| \text{ is an operator norm} \}.
\end{equation}
An operator norm $\|\cdot\|$ is called \emph{extremal} for $\mathcal{M}$, if $\rho(\mathcal{M}) = \|\mathcal{M}\|$. An extremal norm is called a \emph{Barabanov norm}, if in addition for every $x\in \K^n$ there exists an $A \in \mathcal{M}$ such that
\begin{equation*}
    \|Ax\| = \rho(\mathcal{M}) \|x\|. 
\end{equation*}
A sufficient condition for the existence of Barabanov norms is, that the set $\mathcal{M}$ is irreducible, i.e. only the trivial subspaces $\{0\}$ and $\K^d$ are invariant under all $A\in \mathcal{M}$, \cite{Bara88,Wirt02}.

\section{Auerbach bases and normalization}
\label{sec:auerbach}

The main goal of this section is to show that if $\calM\subseteq \K^{d\times d}$ is a compact set of matrices with $\rho(\mathcal{M})>0$, then there exists a similarity transformation $T$ such that all $TAT^{-1}$, $A\in \mathcal{M}$, have entries of absolute value bounded by $\rho(\calM)$. For $\rho(\calM)=0$ this statement is obviously false.
We need some preliminary statements concerning properties of Auerbach bases for this.

Given a norm $\norm{\cdot}$ on $\K^d$ the \emph{dual norm} is defined by
\begin{equation}
    \norm{y}_* := \max \{ |\langle y,x\rangle| \setsep \norm{x} \leq 1 \}, \quad y \in \K^d.
\end{equation}
The following definition can be found e.g. in \cite{taylor1947geometric,weber2017pelczynski}.

\begin{defn}
Let $\norm{\cdot}$ be a norm on $\K^d$. A set of pairs $(x_i,y_i)_{i=1}^d$ in $\K^d\times \K^d$ is called a \emph{biorthogonal basis} of $\K^d$, if $\langle x_i,y_j\rangle = \delta_{ij}$. If, in addition, 
$\norm{x_i} = \norm{y_i}_* =1$, $i=1,\ldots,d$, then $(x_i,y_i)_{i=1}^d$ is an Auerbach or \emph{biorthonormal basis} of $\K^d$.
\end{defn}

It was shown by Auerbach that every finite dimensional normed real or complex space has a biorthonormal basis. We note the following small lemma.

\begin{lem}
\label{lem:Auerbach-transform}
Let $\norm{\cdot}$ be a norm on $\K^d$ with Auerbach basis $(x_i,y_i)_{i=1}^d$. Let $T\in \K^{d\times d}$ with columns $x_i$, $i=1,\ldots, d$. Then $\norm{\cdot}_T := \norm{T\cdot}$ is a norm for which the standard basis vectors form an Auerbach basis $(e_i,e_i)_{i=1}^d$. In particular,
\begin{equation}
\label{eq:vT-bounds}
    \norm{\cdot}_\infty \leq \norm{\cdot}_T \leq \norm{\cdot}_1.
\end{equation}
\end{lem}

\begin{proof}
By construction we have for $i=1,\ldots,d$ that $Te_i = x_i$ and so $\norm{e_i}_T=\norm{x_i} =1$, $i=1,\ldots,d$. In addition,
\begin{align*}
    \norm{e_i}_{T,*} &= \max \{ \abs{\langle e_i,x\rangle} \setsep \norm{x}_T \leq 1 \} = \max \{ |\langle e_i,x\rangle| \setsep \norm{Tx} \leq 1 \} \\
    &= \max \{ \abs{\langle e_i,T^{-1}z\rangle }\setsep \norm{z} \leq 1 \} \\
    &= \max \{ \abs{\langle T^{-\ast}e_i,z\rangle} \setsep \norm{z} \leq 1 \} = 
    \norm{T^{-\ast}e_i}_*
    \intertext{as we have an Auerbach basis, the rows of $T^{-1}$ are given by the $\overline{y}_i$, $i=1,\ldots,d$, so the columns of $T^{-\ast}$ are given by $y_i$, $i=1\ldots,d$, so we can continue}
    &= \norm{y_i}_* = 1.
\end{align*}
In summary, we have $\norm{e_i}_T=\norm{e_i}_{T,*}=1= \langle e_i,e_i\rangle$ 
and also $\langle e_i,e_j\rangle = \delta_{ij}$. This shows that $(e_i,e_i)_{i=1}^d$ is an Auerbach basis corresponding to $\norm{\cdot}_T$.

The left inequality of \eqref{eq:vT-bounds} follows from $\norm{e_i}_{T,*}=1$, $i=1,\ldots,d$. Then for any $x\in\K^d$ we have $\norm{x}_T = \norm{x}_{T,**} \geq |\langle x,e_i\rangle | = |x_i|$. On the other hand $\norm{x}_T \leq \sum_{i=1}^d 
\abs{x_i}\norm{e_i}_T = \norm{x}_1$. This concludes the proof.
\end{proof}

We are now ready to state the main result of this section.

\begin{prop}
\label{prop:diagonal1}
Let $\mathcal{M}\subseteq\K^{d\times d}$ be compact with $\rho(\calM) > 0$. Then there exists a $T\in \gl_d(\K)$ such that 
\begin{equation}
    \rho(\mathcal{M}) \geq \max 
    \{ |a_{ij} | \setsep 1\leq i,j \leq d, A=(a_{ij})_{i,j=1}^d \in T^{-1}\mathcal{M}T \}.
\end{equation}
\end{prop}

\begin{proof}

We first consider the case that $\mathcal{M}$ is irreducible. Then we may choose an extremal norm $\norm{\cdot}$ for $\mathcal{M}$. Let $(x_i,y_i)_{i=1}^d$ be an Auerbach basis for $\norm{\cdot}$ and let $T$ be the invertible matrix from Lemma~\ref{lem:Auerbach-transform} such that the norm $\norm{\cdot}_T$ has an Auerbach basis $(e_i,e_i)_{i=1}^d$. Note that $\norm{\cdot}_T$ is extremal for $T^{-1}\mathcal{M}T$ because for all $A\in\mathcal{M}$ we have
\begin{equation*}
    \norm{T^{-1}AT}_T = \max_{x\neq 0} \frac{\norm{T^{-1}AT x}_T}{\norm{x}_T} = \max_{x\neq 0} \frac{\norm{AT x}}{\norm{Tx}} = \norm{A} \leq \rho(\mathcal{M)} = \rho(T^{-1}\mathcal{M}T).
\end{equation*}
Using \eqref{eq:vT-bounds} we get for arbitrary $A\in T^{-1}\mathcal{M}T$ and $j=1,\ldots,d$ that
\begin{equation}
\label{eq:Auerbachungleichung}
   \rho(\mathcal{M}) = \rho(T^{-1}\mathcal{M}T) \geq \norm{A e_j}_T \geq \|A e_j\|_\infty
   = \max_{i=1,\ldots,d} |a_{ij}|.
\end{equation}
As $j$ was arbitrary, this shows the assertion.

Otherwise, a similarity transformation brings all matrices $A\in \mathcal{M}$ into upper block-triangular form and in this form it is sufficient to consider the irreducible blocks on the diagonal. The entries in the off-diagonal blocks can be made arbitrarily small in absolute value by applying diagonal similarity scaling.
\end{proof}

For an illustration of the previous proof consider the following example which is a modification of \cite[Example 3.1]{guglielmiComputingSpectralGap2023}.
\begin{exam}
Consider  $\calM=\{A_1,A_2\}$
with 
\begin{align*}
    A_1 := \begin{pmatrix}
        6 & -4 \\
        7 & -4
    \end{pmatrix},\quad
    A_2 := \begin{pmatrix}
        -4 & 4 \\
        -5 & 4
    \end{pmatrix}.    
\end{align*}
Using the results presented in \cite[Example 3.1]{guglielmiComputingSpectralGap2023}, we know that the joint spectral radius of $\calM$ is $\rho(\calM)=(48+16\sqrt{5})^{1/5} \approx 2.4245$
and a spectrum maximizing product is given by $B:=A_1A_2A_1^2 A_2$, i.e. $\rho(\calM)^5 = \rho(B)$.
This claim may also be verified by checking that the following procedure generates an extremal norm for $\calM$.
Set $\tilde{A}_i:=\frac{1}{\rho(\calM)} A_i$, $i=1,2$.
%\tilde{A}_2 :=\frac{1}{\rho(\calM)} A_2$.
Then %vector 
$v_1:= (1,\frac{3+\sqrt{5}}{4})^\top$
is a right eigenvector of $B$ for the leading eigenvalue $\rho(\calM)^5$. Define $v_2:=\tilde{A}_2v_1$, $v_3:=\tilde{A}_1v_2$, $v_4:=\tilde{A}_1v_3$, $v_5:=\tilde{A}_2v_4$, $v_6:=\tilde{A}_2v_3$.
The unit ball $K$ of an extremal norm 
for $\calM$ is given by the convex hull
of $\{v_1,\dots,v_6,-v_1,\dots,-v_6\}$.
An Auerbach basis is given by
$(v_3,v_6)$ together with its dual basis, see Figure \ref{fig:auerbach-basis}.
\begin{figure}[h]
    \begin{center}
    \begin{tikzpicture}[scale=1.7]
\pgfdeclarelayer{foreground}
\pgfsetlayers{main,foreground}
%% Drawing the axes
\draw[color=black,-latex] (-1.50000000000000,0) -- (1.50000000000000,0);
\draw[color=black,-latex] (0,-1.50000000000000) -- (0,1.50000000000000);
\draw (-1,0.03) -- (-1,-0.03) node[below] {-1};
\draw (1,0.03) -- (1,-0.03) node[below] {1};
\draw (0.03,-1) -- (-0.03,-1) node[left] {-1};
\draw (0.03,1) -- (-0.03,1) node[left] {1};
\fill[color=Gray, opacity=0.100000000000000] (-1.447944167, -1.204099468) --(-0.7541075963, -1.41850676) --(1.447944167, 1.204099468) --(0.7541075963, 1.41850676) --cycle {};
\draw[color=Black, opacity =1.0] (-0.7541075963, -1.41850676) -- (1.447944167, 1.204099468);
\draw[color=Black, opacity =1.0] (-0.7541075963, -1.41850676) -- (-1.447944167, -1.204099468);
\draw[color=Black, opacity =1.0] (1.447944167, 1.204099468) -- (0.7541075963, 1.41850676);
\draw[color=Black, opacity =1.0] (0.7541075963, 1.41850676) -- (-1.447944167, -1.204099468);
\fill[color=CornflowerBlue, opacity=0.900000000000000] (1.0, 1.309016994) --(0.5613255774, 1.015447509) --(-0.3469182855, 0.1072036459) --(-0.5098205118, -0.09736705484) --(-0.7492166416, -0.5176959652) --(-1.101025882, -1.311303114) --(-1.0, -1.309016994) --(-0.5613255774, -1.015447509) --(0.3469182855, -0.1072036459) --(0.5098205118, 0.09736705484) --(0.7492166416, 0.5176959652) --(1.101025882, 1.311303114) --cycle {};
\draw[color=Black, opacity =1.0] (-1.0, -1.309016994) -- (-1.101025882, -1.311303114);
\draw[color=Black, opacity =1.0] (-1.0, -1.309016994) -- (-0.5613255774, -1.015447509);
\draw[color=Black, opacity =1.0] (-1.101025882, -1.311303114) -- (-0.7492166416, -0.5176959652);
\draw[color=Black, opacity =1.0] (0.3469182855, -0.1072036459) -- (0.5098205118, 0.09736705484);
\draw[color=Black, opacity =1.0] (0.3469182855, -0.1072036459) -- (-0.5613255774, -1.015447509);
\draw[color=Black, opacity =1.0] (0.5098205118, 0.09736705484) -- (0.7492166416, 0.5176959652);
\draw[color=Black, opacity =1.0] (0.7492166416, 0.5176959652) -- (1.101025882, 1.311303114);
\draw[color=Black, opacity =1.0] (-0.7492166416, -0.5176959652) -- (-0.5098205118, -0.09736705484);
\draw[color=Black, opacity =1.0] (0.5613255774, 1.015447509) -- (-0.3469182855, 0.1072036459);
\draw[color=Black, opacity =1.0] (0.5613255774, 1.015447509) -- (1.0, 1.309016994);
\draw[color=Black, opacity =1.0] (-0.5098205118, -0.09736705484) -- (-0.3469182855, 0.1072036459);
\draw[color=Black, opacity =1.0] (1.101025882, 1.311303114) -- (1.0, 1.309016994);
\begin{pgfonlayer}{foreground}
\node[circle, inner sep = 0.300000000000000 mm, fill, draw, color=Blue, opacity=1.0] (0) at (-1.0, -1.309016994) {};
\node[circle, inner sep = 0.300000000000000 mm, fill, draw, color=Blue, opacity=1.0] (1) at (-1.101025882, -1.311303114) {};
\node[circle, inner sep = 0.300000000000000 mm, fill, draw, color=Blue, opacity=1.0] (2) at (0.3469182855, -0.1072036459) {};
\node[circle, inner sep = 0.300000000000000 mm, fill, draw, color=Blue, opacity=1.0] (3) at (0.5098205118, 0.09736705484) {};
\node[circle, inner sep = 0.300000000000000 mm, fill, draw, color=Blue, opacity=1.0] (4) at (-0.5613255774, -1.015447509) {};
\node[circle, inner sep = 0.300000000000000 mm, fill, draw, color=Blue, opacity=1.0] (5) at (0.7492166416, 0.5176959652) {};
\node[circle, inner sep = 0.300000000000000 mm, fill, draw, color=Blue, opacity=1.0] (6) at (-0.7492166416, -0.5176959652) {};
\node[circle, inner sep = 0.300000000000000 mm, fill, draw, color=Blue, opacity=1.0] (7) at (0.5613255774, 1.015447509) {};
\node[circle, inner sep = 0.300000000000000 mm, fill, draw, color=Blue, opacity=1.0] (8) at (-0.5098205118, -0.09736705484) {};
\node[circle, inner sep = 0.300000000000000 mm, fill, draw, color=Blue, opacity=1.0] (9) at (-0.3469182855, 0.1072036459) {};
\node[circle, inner sep = 0.300000000000000 mm, fill, draw, color=Blue, opacity=1.0] (10) at (1.101025882, 1.311303114) {};
\node[circle, inner sep = 0.300000000000000 mm, fill, draw, color=Blue, opacity=1.0] (11) at (1.0, 1.309016994) {};
\end{pgfonlayer}n\fill[color=Gray, opacity=0.300000000000000] (-0.3469182854, 0.107203646) --(1.101025882, 1.311303114) --(0.3469182854, -0.107203646) --(-1.101025882, -1.311303114) --cycle {};
\draw[color=Black, opacity =1.0] (1.101025882, 1.311303114) -- (0.3469182854, -0.107203646);
\draw[color=Black, opacity =1.0] (1.101025882, 1.311303114) -- (-0.3469182854, 0.107203646);
\draw[color=Black, opacity =1.0] (0.3469182854, -0.107203646) -- (-1.101025882, -1.311303114);
\draw[color=Black, opacity =1.0] (-1.101025882, -1.311303114) -- (-0.3469182854, 0.107203646);
\draw[-latex, Black!90, very thick] (0,0) -- (1.10102588180999, 1.31130311398904) node[above right] {$v_3$};
 \draw[-latex, Black!90, very thick] (0,0) -- (0.346918285537653, -0.107203645890555) node[below right] {$v_6$};
 \end{tikzpicture}
    \end{center}
    \caption{The unit ball $K$ of an extremal norm.}
    \label{fig:auerbach-basis}
\end{figure}
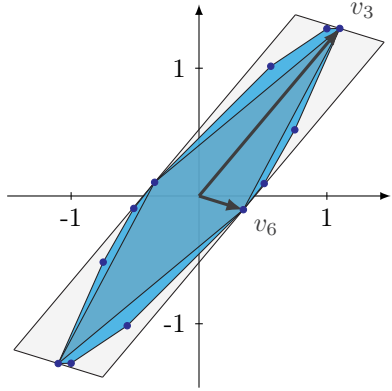

Set $T:=\begin{pmatrix} v_3 & v_6\end{pmatrix}$. Then $T^{-1}A_iT$, $i=1,2$ are given approximately by
\begin{align*}
        %T^{-1}A_1T &\approx 
        \begin{pmatrix}
            1.7454 & 2.1998 \\
            -1.6163 & 0.2546
        \end{pmatrix}, %\\
        \quad
        %T^{-1}A_2T &\approx 
        \begin{pmatrix}
        0 & -1.6498 \\
        2.4245 & 0
        \end{pmatrix}.
\end{align*}
The absolute values of all entries of these two transformed matrices
are now bounded by $\rho(\calM)$. See Figure \ref{fig:auerbach-basis-transformed}
\begin{figure}
    \begin{center}
%%%%%%%%%%%%%%%%%%%%%%%%%%%%%%%%%%%%%%%%%%%%%%%%%%%%%%%
\begin{tikzpicture}[scale=1.70000000000000]
\pgfdeclarelayer{foreground}
\pgfsetlayers{main,foreground}
%% Drawing the axes
\draw[color=black,-latex] (-1.10000000000000,0) -- (1.10000000000000,0);
\draw[color=black,-latex] (0,-1.10000000000000) -- (0,1.10000000000000);
\fill[color=Gray, opacity=0.100000000000000] (1, 1) --(-1, 1) --(-1, -1) --(1, -1) --cycle {};
\draw[color=Black, opacity =1.0] (-1, -1) -- (-1, 1);
\draw[color=Black, opacity =1.0] (-1, -1) -- (1, -1);
\draw[color=Black, opacity =1.0] (-1, 1) -- (1, 1);
\draw[color=Black, opacity =1.0] (1, -1) -- (1, 1);
\fill[color=CornflowerBlue, opacity=0.900000000000000] (-0.97971296, 0.2268238065) --(-0.7198781136, 0.6666666665) --(0.0, 1.0) --(0.1543471174, 0.97971296) --(0.4536476109, 0.7198781143) --(1.0, 0.0) --(0.97971296, -0.2268238065) --(0.7198781136, -0.6666666665) --(0.0, -1.0) --(-0.1543471174, -0.97971296) --(-0.4536476109, -0.7198781143) --(-1.0, 0.0) --cycle {};
\draw[color=Black, opacity =1.0] (0.97971296, -0.2268238065) -- (1.0, 0.0);
\draw[color=Black, opacity =1.0] (0.97971296, -0.2268238065) -- (0.7198781136, -0.6666666665);
\draw[color=Black, opacity =1.0] (1.0, 0.0) -- (0.4536476109, 0.7198781143);
\draw[color=Black, opacity =1.0] (0.0, 1.0) -- (0.1543471174, 0.97971296);
\draw[color=Black, opacity =1.0] (0.0, 1.0) -- (-0.7198781136, 0.6666666665);
\draw[color=Black, opacity =1.0] (0.1543471174, 0.97971296) -- (0.4536476109, 0.7198781143);
\draw[color=Black, opacity =1.0] (0.7198781136, -0.6666666665) -- (0.0, -1.0);
\draw[color=Black, opacity =1.0] (-0.4536476109, -0.7198781143) -- (-0.1543471174, -0.97971296);
\draw[color=Black, opacity =1.0] (-0.4536476109, -0.7198781143) -- (-1.0, 0.0);
\draw[color=Black, opacity =1.0] (-0.7198781136, 0.6666666665) -- (-0.97971296, 0.2268238065);
\draw[color=Black, opacity =1.0] (-0.1543471174, -0.97971296) -- (0.0, -1.0);
\draw[color=Black, opacity =1.0] (-1.0, 0.0) -- (-0.97971296, 0.2268238065);
\begin{pgfonlayer}{foreground}
\node[circle, inner sep = 0.300000000000000 mm, fill, draw, color=Blue, opacity=1.0] (0) at (0.97971296, -0.2268238065) {};
\node[circle, inner sep = 0.300000000000000 mm, fill, draw, color=Blue, opacity=1.0] (1) at (1.0, 0.0) {};
\node[circle, inner sep = 0.300000000000000 mm, fill, draw, color=Blue, opacity=1.0] (2) at (0.0, 1.0) {};
\node[circle, inner sep = 0.300000000000000 mm, fill, draw, color=Blue, opacity=1.0] (3) at (0.1543471174, 0.97971296) {};
\node[circle, inner sep = 0.300000000000000 mm, fill, draw, color=Blue, opacity=1.0] (4) at (0.7198781136, -0.6666666665) {};
\node[circle, inner sep = 0.300000000000000 mm, fill, draw, color=Blue, opacity=1.0] (5) at (0.4536476109, 0.7198781143) {};
\node[circle, inner sep = 0.300000000000000 mm, fill, draw, color=Blue, opacity=1.0] (6) at (-0.4536476109, -0.7198781143) {};
\node[circle, inner sep = 0.300000000000000 mm, fill, draw, color=Blue, opacity=1.0] (7) at (-0.7198781136, 0.6666666665) {};
\node[circle, inner sep = 0.300000000000000 mm, fill, draw, color=Blue, opacity=1.0] (8) at (-0.1543471174, -0.97971296) {};
\node[circle, inner sep = 0.300000000000000 mm, fill, draw, color=Blue, opacity=1.0] (9) at (0.0, -1.0) {};
\node[circle, inner sep = 0.300000000000000 mm, fill, draw, color=Blue, opacity=1.0] (10) at (-1.0, 0.0) {};
\node[circle, inner sep = 0.300000000000000 mm, fill, draw, color=Blue, opacity=1.0] (11) at (-0.97971296, 0.2268238065) {};
\end{pgfonlayer}n\fill[color=Gray, opacity=0.300000000000000] (1, 0) --(0, -1) --(-1, 0) --(0, 1) --cycle {};
\draw[color=Black, opacity =1.0] (-1, 0) -- (0, -1);
\draw[color=Black, opacity =1.0] (-1, 0) -- (0, 1);
\draw[color=Black, opacity =1.0] (0, -1) -- (1, 0);
\draw[color=Black, opacity =1.0] (0, 1) -- (1, 0);
\end{tikzpicture}
%%%%%%%%%%%%%%%%%%%%%%%%%%%%%%%%%%%%
    \end{center}
    \caption{The transformed unit ball $T^{-1}K$.}
    \label{fig:auerbach-basis-transformed}
\end{figure}
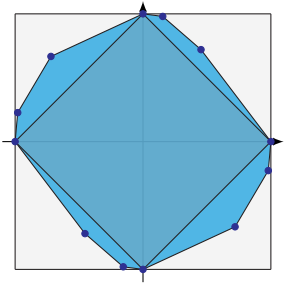
for an illustration of the transformed unit 
ball $T^{-1}K$ which clearly lies between the unit ball
of the $1$- and the $\infty$-norm.
\end{exam}

With the previous result we obtain the following characterization of the joint spectral radius which is close in spirit to \cite{chen2000characterization}.

\begin{prop}
Let $\mathcal{M}\subseteq\K^{d\times d}$ be compact. Then there exists a $T\in \gl_d(\K)$ such that 
\begin{align}
\label{eq:diagonal_char}
    \rho(\mathcal{M}) &= \sup_{t\geq 1} \max \{ |\langle e_i,S e_i \rangle|^{1/t} \; | \;  i=1,\ldots,d, S \in T^{-1}\mathcal{S}_tT \} \\
    &= \limsup_{t \to \infty} \max \{ |\langle e_i,S e_i \rangle|^{1/t} \; | \;  i=1,\ldots,d, S \in T^{-1}\mathcal{S}_tT \}.
\end{align}
\end{prop}

\begin{proof}
If $\rho(\calM)=0$, then $\calM$ is reducible, see \cite[Section~2.3.1]{Jungers}, and after a similarity transformation all matrices $A\in T^{-1}\calM T$ are upper triangular with zero diagonal. It follows by induction that all $S \in T^{-1}\mathcal{S}_tT$ have the same structure and the claim is obvious.
 We may thus assume $\rho(\calM) >0$.

If $\rho(\calM) >0$, we assume without loss of generality that $\calM$ is irreducible.
As in the proof of \Cref{prop:diagonal1} fix $T$ using \Cref{lem:Auerbach-transform}.
We may apply the chain of inequalities in \eqref{eq:Auerbachungleichung}
to all $S\in T^{-1}\mathcal{S}_tT$. 
As $\norm{\cdot}_T$ is extremal for $T^{-1}\calM T$, it is also extremal for $T^{-1}\mathcal{S}_tT$.
This yields
\begin{equation*}
    \rho(\calM)^t \geq \norm{Se_i}_T \geq |\langle e_i,S e_i \rangle|, \quad i=1,\ldots,d, t\geq 1, S \in T^{-1}\mathcal{S}_tT.
\end{equation*}
This shows $\geq$ in \eqref{eq:diagonal_char}. 

For the converse recall that by \cite{chen2000characterization} we know that
\begin{multline*}
    \rho(\mathcal{M}) = \rho(T^{-1} \calM T) = \limsup_{t\to\infty} \max_{S\in T^{-1}\mathcal{S}_tT} |\trace S|^{1/t} = \\ \limsup_{t\to\infty} \max_{S\in T^{-1}\mathcal{S}_tT} \left|\frac{1}{d}\trace S\right|^{1/t} 
    \leq
    \limsup_{t \to \infty} \max_{S\in T^{-1}\mathcal{S}_tT} \max_{i=1,\ldots,d}\left|\langle e_i, Se_i\rangle\right|^{1/t}.
\end{multline*}
This completes the proof.
\end{proof}

\section{Higher dimensional principal
  submatrices}
\label{sec:higher-dimensionalPS}

It seems natural to ask for a generalization
of the results of the previous section to 
larger submatrices. The question is then whether we
can use a similarity transformation to
ensure that the joint spectral
radius of all principal submatrices
is bounded by the joint spectral radius of the whole matrix set.
The aim of this section is to show, that
this is in general only possible
up to a multiplicative constant and 
that this question has intricate connections
to convex geometry and the local theory of Banach spaces.

We begin by fixing some necessary notation. Let $J \subseteq \{1,\dots,d\}$ be a set of indices of cardinality $k\geq 1$.
For a matrix $A \in \mathbb{K}^{d \times d}$
we denote by
$A_{J,J}$ the principal submatrix of $A$ obtained
by deleting all rows and columns whose index is not in $J$.
Similarly, for a set of matrices $\mathcal{M} \subseteq\mathbb{K}^{d \times d}$
we define $\mathcal{M}_{J,J} := \{A_{J,J} \setsep A \in \mathcal{M}\}$.
We denote by $I_J$ the matrix obtained from the $d \times d$
  identity matrix by deleting all columns whose index is not in $J$.
  For any matrix $A \in \mathbb{K}^{d \times d}$ we have $A_{J,J}=I_J^\top A I_J$.
  Furthermore, $I_J^\top I_J$ is the $k \times k$ identity matrix
  and $P_J:= I_J I_J^\top$ is the projection
  to the image of $I_J$ whose kernel is spanned
  by the columns of $I_{\{ 1,\dots,d\} \setminus J}$.
  
  For a norm $\norm{\cdot}$ denote
  by $B_{r,\norm{\cdot}}$ the closed ball of radius $r$ around $0$
  with respect to $\norm{\cdot}$.

We start with a negative result which shows that 
the direct generalization of 
\Cref{prop:diagonal1}
to larger submatrices is false.

\begin{thm}
  \label{thm:lower-bound-upper-left}
  For every $d \geq 3$ there is a finite set of matrices $\calM$ with
  $\rho(\calM)=1$ and a constant $C>1$,
  such that for all $T \in \gl_d(\R)$ and for all
  $J \subseteq \{1,\dots,d\}$ with $\abs{J} \in \{2,\dots,d-1\}$ we
  have $\rho((T^{-1} \calM T)_{J,J}) \geq C$.
\end{thm}

We postpone the proof and first present some
 key ingredients. One of these 
is the following result of Bosznay and Garay.
To simplify notation we make the following definition.
\begin{defn}
\label{defn:shady}
For $C>1$ we call a norm $\norm{\cdot}$ on $\R^d$ \emph{$C$-shady}\footnote{As far as we are aware, there is no established name for this class of norms. We picked the term \enquote{shady} for
the following reason: The unit ball of such a norm always
casts a shadow on every proper linear subspace of dimension at least two, regardless of where the light is coming from.}  if for every
  projection $P: \R^d \to \R^d$ with $\rank(P) \in \{2,\dots,d-1\}$
  we have for the induced norm that $\norm{P}\geq C$.
  We call a norm \emph{shady} if it is $C$-shady for some $C>1$.
\end{defn}

Note that by the Hahn-Banach theorem for every norm $\|\cdot\|$ on $\R^d$ and every one-dimensional subspace $V$ there is a projection onto $V$ of norm $1$. Also trivially the only projection of rank $d$ has norm $1$. Therefore the rank restriction in Definition~\ref{defn:shady} captures all cases where something interesting can happen.

For the next statement recall that a {\em polytopal norm} on $\R^d$ is a norm whose unit ball is a polytope.

\begin{thm}[Bosznay and Garay 1986 \cite{bosznayNormsProjections1986}]
  \label{lem:bosznay-garay}
  For every $d\geq 3$ there is a shady polytopal norm on $\R^d$.
\end{thm}

There are earlier examples of shady norms. For instance
Singer in \cite[Chapter II, Theorem 1.1, p. 217]{singerBasesBanachSpaces1970}
shows for a particular norm on $\R^3$, whose unit ball
is a dodecahedron with
suitably cut-off vertices, that it is shady.  Bosznay and Garay showed
that the shady normed spaces are actually dense in the
space of norms, considered as bounded continuous functions
on the Euclidean unit ball, see \cite{bosznayNormsProjections1986} for details. For explicit examples of $C$-shady norms $\norm{\cdot}$ and bounds
on the constant $C$ see the work of Kobos in \cite{kobosUniformEstimateRelative2018}
and \cite{kobosHYPERPLANESFINITEDIMENSIONALNORMED2015}.

We use the following lemma to construct a specific matrix set
tailored to a shady norm. 

\begin{lem}
  \label{lem:taylored-matrix-sets-exist}
  Let $K$ be the unit ball of a polytopal norm $\norm{\cdot}$ in $\R^d$.
  For every vertex $v$ and face $F$ of $K$ there is a
  rank one matrix $A$ with operator norm $\norm{A} = 1$ which
  maps $F$ to $v$.
\end{lem}
\begin{proof}
  Let $F$ be a face of $K$ and let $v$ be a vertex of $K$.
  Let $w$ be a vertex of $F$. Let $H$ be a supporting
  hyperplane of $F$, i.e. $(H+w) \cap K = F$.
  Let $\varphi: \R^d \to \R$
  be the linear functional whose kernel is $H$
  and which is $1$ on $w$.
  Then $\abs{\varphi(x)} \leq \varphi(y)=1$
  for all $x \in K$ and $y \in F$.

  Let $A$ be the rank one matrix defined
  by $Ax=\varphi(x)v$. 
  We have \[AK \subseteq \{\alpha v \setsep \alpha \in [-1,1]\} \subseteq K\]
  and for $y \in F$ we have 
  $Ay=\varphi(y)v = v$.
  Hence $A$ is a matrix with operator norm $\norm{A}=1$
  which maps $F$ to $v$.
  \end{proof}

\begin{proof}[Proof of \Cref{thm:lower-bound-upper-left}]
  Let $\norm{\cdot}$ be a $C$-shady polytopal norm on $\R^d$ for some $C>1$.
  Such a norm exists by \Cref{lem:bosznay-garay}. Let $K$ be the
  unit ball of $\norm{\cdot}$.
  For every vertex $v$ of $K$ and every
  face $F$ of $K$ let
  $A_{v,F}$ be a matrix with operator norm $\norm{A_{v,F}}=1$
  which maps $F$ to $v$.
  The existence of these maps is ensured by
  \Cref{lem:taylored-matrix-sets-exist}.
  Define
  \begin{align*}
    \calM := \{A_{v,F} \setsep \text{$v$ is a vertex of $K$, $F$ is a
    facet of $K$}\}.
  \end{align*}
  Since $K$ is a polytope, the collection of matrices $\calM$
  is finite. Also it is easy to see that $\calM$ is irreducible.
  
  Let $T \in \R^{d \times d}$ be invertible
  and let $w_1,\dots,w_d$ be the columns of $T$.
  Let $J$ be an index set with cardinality between $2$ and $d-1$.
  Let $W$ be the subspace spanned by $w_j, j \in J$
  and let $U$ be the subspace spanned by $w_i, i \not \in J$,
  thus $\R^d = W \oplus U$.
  Let $Q$ be the projection on $\R^d$ with image $W$ and kernel $U$.
  Notice that $T^{-1}QT=P_J$.
  By the assumption on $\norm{\cdot}$ we have $\norm{Q} \geq C$.
  Since $K$ is a polytope, there is a vertex $v$ of $K$ with $\alpha:=\norm{Qv}\geq C$.
  There is a facet $F$ of $K$ such that $\frac{1}{\alpha}Qv \in F$.
  Then $A_{v,F}\frac{1}{\alpha}Qv=v$.
  For $w:=Qv$ we obtain $QA_{v,F}w=\alpha Qv=\alpha w$.
  Since $W$ is spanned by the columns of $TI_J$, there is $z \in \R^{\abs{J}}$ such that $T I_J z=w$.
  Hence
  \begin{align*}
    (T^{-1}A_{v,F} T)_{J,J} z &= I_J^\top T^{-1} A_{v,F} T I_J z \\
                              &=I_J^\top T^{-1} A_{v,F} w \\
                              &=I_J^\top I_J I_J^\top T^{-1} A_{v,F} w \\
                              &=I_J^\top T^{-1} Q A_{v,F} w \\
                              &=\alpha I_J^\top T^{-1} w \\
                              &=\alpha z.
  \end{align*}
  This show that
  $\rho(T^{-1}\calM T)_{J,J} \geq \rho((T^{-1}A_{v,F} T)_{J,J})\geq
  \alpha \geq C$.
\end{proof}

\begin{rem}
    \label{rem:shadinessconstant} In relation to \Cref{thm:lower-bound-upper-left}, we can consider the following "submatrix constant" for dimension $d$. For a fixed dimension $d\geq 3$ consider the supremum of the values $C$  for which there exists a finite matrix set $\calM \subseteq \R^{d\times d}$, $\rho(\calM)=1$, such that for all $T\in \gl_d(\R)$, for all $J\subseteq \{1,\ldots,d\}$, $|J|\in \{2,\ldots,d-1\}$ we have $\rho(((T^{-1} \calM T)_{J,J}) \geq C$. The proof of \Cref{thm:lower-bound-upper-left} shows directly, that for any $C$-shady polytopal norm we can construct an associated finite, irreducible matrix set with constant $C$. So the submatrix constant is bigger or equal than the supremal polytopal shadiness in dimension $d$.

    On the other hand let $\calM$ be a finite irreducible set of matrices with 
    $\rho(\calM)=1$ such that for all $T\in \gl_d(\R)$ and for all $J\subseteq \{1,\ldots,d\}$, $|J|\in \{2,\ldots,d-1\}$ we have $\rho(((T^{-1} \calM T)_{J,J}) \geq C$. Let $\norm{\cdot}$ be an extremal norm 
    for $\calM$ (which is not necessarily polytopal) and let $P$ be a rank $k$ projection with $2 \leq k \leq d-1$.
    There is $T \in \gl_d(\R)$ such that 
    $T^{-1}PT$ is the projection onto the first $k$ coordinates.
    Then $(T^{-1}\calM T)_{J,J}$ with $J=\{1,\dots,k\}$
    has joint spectral radius at most $\norm{T^{-1}PT}_T=\norm{P}$.
    Hence $\norm{P} \geq C$. This shows that $\norm{\cdot}$ is $C$-shady. Therefore the supremal shadiness in dimension $d$ is larger or equal to the submatrix constant.
     We omit the details on the relation of supremal polytopal shadiness and supremal shadiness and point to \cite{epperlein2025shadiness} for a proof that these constants are actually the same. 
\end{rem}

In the following example we construct a finite set $\calM$ of real $3\times 3$
matrices with joint spectral radius $1$, for which the set of upper left
$2 \times 2$ principal submatrices of $T^{-1}\calM T$ have joint spectral radius at least $1.01$ for every $T \in \gl_3(\R)$.

\begin{exam}
\label{exam:icosahedron}
  Consider the following set of points in $\R^3$,
  \begin{align*}
    V = \left\{
    \vect{1\\a\\c},
    \vect{1\\b\\c},
    \vect{a\\c\\1},
    \vect{b\\c\\1},
    \vect{c\\1\\a},
    \vect{c\\1\\b}\right\}
  \end{align*}
  for $a = -0.6$, $b=-0.2$, $c=0.1$.
  Let $K$ be the polytope with vertices $V \cup -V$ depicted in
  \Cref{fig:icosahedron}.
  Combinatorially this polytope is an icosahedron.
  Since it is centrally symmetric, it
  is the unit ball of a norm on $\R^3$.
  
  Finding the minimal norm of a rank $2$-projection
  in $X$
  can be formulated as a nonconvex nonlinear optimization
  problem with linear objective function and quadratic constraints.
  Using for example the open source SCIP Optimization suite\footnote{Notice
    that SCIP can find global optima even
    for nonconvex problems with quadratic constraints
    using branch-and-bound-techniques.} \cite{BolusaniEtal2024OO}
  we can show numerically that this minimum in our case is at least $1.01$. In forthcoming work, \cite{epperlein2025shadiness}, we will show how to obtain 
  computer-assisted rigorous proofs for this lower bound.

  Since $V$ is centrally symmetric and invariant under cyclic permutations of the coordinates, $K$ is invariant
  under multiplication by
  \begin{align*}
    A=\begin{pmatrix}0&0&-1\\-1&0&0\\0&-1&0 \end{pmatrix}.
  \end{align*}
  Under the action of $A$ their are two vertex-orbits
  and four facet-orbits.

  Consider
  \begingroup
    \renewcommand*{\arraystretch}{1.2}
  \begin{align*}
    W=\left\{  
    \vect{
    \frac{20}{53} \\
    \frac{-55}{53} \\
    0
    },
    \vect{
    \frac{20}{17} \\
    \frac{15}{17} \\
    0 \\
    },
    \vect{
    \frac{10}{9} \\
    \frac{10}{9} \\
    \frac{10}{9} \\
    },
    \vect{
    \frac{390}{1069} \\
    \frac{-1250}{1069} \\
    \frac{-710}{1069} \\
    }
    \right\}.
  \end{align*}
  \endgroup
  The orbits of $W$ under multiplication by $A$
  are the normals of the facets of $P$, scaled such that
  the scalar product with the vertices of the corresponding facet
  equals $1$.

  Finally set
  \begin{align*}
    \calM := \{A^k vw^\top A^\ell \setsep k,\ell \in \{0,\dots,5\}, v \in
    V, w \in W\}.
  \end{align*}
  All of the matrices in $\calM$ map $P$ to itself
  and for each facet $F$ of $P$ and each vertex $v$ of $P$
  there is a matrix in $\calM$ mapping
  $F$ to $v$, so $\rho(\calM)=1$.
  On the other hand the proof of \Cref{thm:lower-bound-upper-left} applied to our example
  shows that $\rho((T^{-1}\calM T)_{J,J})\geq
  1.01$
  for all $T \in \gl_3(\R)$ and $J \subseteq \{1,2,3\}, \abs{J}=2$.
  
  \begin{figure}
    \begin{center}
    \scalebox{1.0}{
\begin{tikzpicture}%
	[x={(0.481741cm, -0.248942cm)},
	y={(0.876313cm, 0.136837cm)},
	z={(0.000014cm, 0.958803cm)},
	scale=3.000000,
	back/.style={line cap=round, dash pattern=on 0pt off 5\pgflinewidth, thin},
	edge/.style={color=black, thick},
	facet/.style={fill=CornflowerBlue},
	vertex/.style={inner sep=0.7pt,circle,draw=black!25!black,fill=black!75!black,thick}]
%
%
%% This TikZ-picture was produced with Sagemath version 10.4
%% with the command: ._tikz_3d_in_3d and parameters:
%% view = [-0.711200000000000, -0.420600000000000, -0.563300000000000]
%% angle = 92.7900000000000
%% scale = 3
%% edge_color = black
%% facet_color = OliveGreen
%% opacity = 0.100000000000000
%% vertex_color = black
%% axis = False
%%
%% Coordinate of the vertices:
%%
\coordinate (-1.00000, 0.20000, -0.10000) at (-1.00000, 0.20000, -0.10000);
\coordinate (-1.00000, 0.60000, -0.10000) at (-1.00000, 0.60000, -0.10000);
\coordinate (-0.60000, 0.10000, 1.00000) at (-0.60000, 0.10000, 1.00000);
\coordinate (-0.20000, 0.10000, 1.00000) at (-0.20000, 0.10000, 1.00000);
\coordinate (-0.10000, -1.00000, 0.20000) at (-0.10000, -1.00000, 0.20000);
\coordinate (-0.10000, -1.00000, 0.60000) at (-0.10000, -1.00000, 0.60000);
\coordinate (0.10000, 1.00000, -0.60000) at (0.10000, 1.00000, -0.60000);
\coordinate (0.10000, 1.00000, -0.20000) at (0.10000, 1.00000, -0.20000);
\coordinate (0.20000, -0.10000, -1.00000) at (0.20000, -0.10000, -1.00000);
\coordinate (0.60000, -0.10000, -1.00000) at (0.60000, -0.10000, -1.00000);
\coordinate (1.00000, -0.60000, 0.10000) at (1.00000, -0.60000, 0.10000);
\coordinate (1.00000, -0.20000, 0.10000) at (1.00000, -0.20000, 0.10000);
%%
%%
%% Drawing edges in the back
%%
\draw[edge,back] (-1.00000, 0.20000, -0.10000) -- (-1.00000, 0.60000, -0.10000);
\draw[edge,back] (-1.00000, 0.20000, -0.10000) -- (-0.60000, 0.10000, 1.00000);
\draw[edge,back] (-1.00000, 0.20000, -0.10000) -- (-0.10000, -1.00000, 0.20000);
\draw[edge,back] (-1.00000, 0.20000, -0.10000) -- (-0.10000, -1.00000, 0.60000);
\draw[edge,back] (-1.00000, 0.20000, -0.10000) -- (0.20000, -0.10000, -1.00000);
\draw[edge,back] (-1.00000, 0.60000, -0.10000) -- (-0.60000, 0.10000, 1.00000);
\draw[edge,back] (-1.00000, 0.60000, -0.10000) -- (0.10000, 1.00000, -0.60000);
\draw[edge,back] (-1.00000, 0.60000, -0.10000) -- (0.10000, 1.00000, -0.20000);
\draw[edge,back] (-1.00000, 0.60000, -0.10000) -- (0.20000, -0.10000, -1.00000);
\draw[edge,back] (-0.60000, 0.10000, 1.00000) -- (0.10000, 1.00000, -0.20000);
\draw[edge,back] (0.10000, 1.00000, -0.60000) -- (0.20000, -0.10000, -1.00000);
%%
%%
%% Drawing vertices in the back
%%
\node[vertex,draw=black!0,inner sep=0.7pt] at (-1.00000, 0.20000, -0.10000)     {};
\node[vertex,draw=black!0,inner sep=0.7pt] at (-1.00000, 0.60000, -0.10000)     {};
%%
%%
%% Drawing the facets
%%
\fill[facet, opacity=0.22] (-0.10000, -1.00000, 0.60000) -- (-0.60000, 0.10000, 1.00000) -- (-0.20000, 0.10000, 1.00000) -- cycle {};
\fill[facet, opacity=0.50] (0.60000, -0.10000, -1.00000) -- (-0.10000, -1.00000, 0.20000) -- (0.20000, -0.10000, -1.00000) -- cycle {};
\fill[facet, opacity=0.52] (1.00000, -0.60000, 0.10000) -- (-0.10000, -1.00000, 0.20000) -- (0.60000, -0.10000, -1.00000) -- cycle {};
\fill[facet, opacity=0.39] (1.00000, -0.60000, 0.10000) -- (-0.10000, -1.00000, 0.20000) -- (-0.10000, -1.00000, 0.60000) -- cycle {};
\fill[facet, opacity=0.39] (1.00000, -0.60000, 0.10000) -- (-0.20000, 0.10000, 1.00000) -- (-0.10000, -1.00000, 0.60000) -- cycle {};
\fill[facet, opacity=0.50] (1.00000, -0.20000, 0.10000) -- (-0.20000, 0.10000, 1.00000) -- (1.00000, -0.60000, 0.10000) -- cycle {};
\fill[facet, opacity=0.88] (1.00000, -0.20000, 0.10000) -- (0.60000, -0.10000, -1.00000) -- (1.00000, -0.60000, 0.10000) -- cycle {};
\fill[facet, opacity=0.96] (1.00000, -0.20000, 0.10000) -- (0.10000, 1.00000, -0.60000) -- (0.60000, -0.10000, -1.00000) -- cycle {};
\fill[facet, opacity=0.70] (1.00000, -0.20000, 0.10000) -- (-0.20000, 0.10000, 1.00000) -- (0.10000, 1.00000, -0.20000) -- cycle {};
\fill[facet, opacity=0.91] (1.00000, -0.20000, 0.10000) -- (0.10000, 1.00000, -0.60000) -- (0.10000, 1.00000, -0.20000) -- cycle {};
%%
%%
%% Drawing edges in the front
%%
\draw[edge] (-0.60000, 0.10000, 1.00000) -- (-0.20000, 0.10000, 1.00000);
\draw[edge] (-0.60000, 0.10000, 1.00000) -- (-0.10000, -1.00000, 0.60000);
\draw[edge] (-0.20000, 0.10000, 1.00000) -- (-0.10000, -1.00000, 0.60000);
\draw[edge] (-0.20000, 0.10000, 1.00000) -- (0.10000, 1.00000, -0.20000);
\draw[edge] (-0.20000, 0.10000, 1.00000) -- (1.00000, -0.60000, 0.10000);
\draw[edge] (-0.20000, 0.10000, 1.00000) -- (1.00000, -0.20000, 0.10000);
\draw[edge] (-0.10000, -1.00000, 0.20000) -- (-0.10000, -1.00000, 0.60000);
\draw[edge] (-0.10000, -1.00000, 0.20000) -- (0.20000, -0.10000, -1.00000);
\draw[edge] (-0.10000, -1.00000, 0.20000) -- (0.60000, -0.10000, -1.00000);
\draw[edge] (-0.10000, -1.00000, 0.20000) -- (1.00000, -0.60000, 0.10000);
\draw[edge] (-0.10000, -1.00000, 0.60000) -- (1.00000, -0.60000, 0.10000);
\draw[edge] (0.10000, 1.00000, -0.60000) -- (0.10000, 1.00000, -0.20000);
\draw[edge] (0.10000, 1.00000, -0.60000) -- (0.60000, -0.10000, -1.00000);
\draw[edge] (0.10000, 1.00000, -0.60000) -- (1.00000, -0.20000, 0.10000);
\draw[edge] (0.10000, 1.00000, -0.20000) -- (1.00000, -0.20000, 0.10000);
\draw[edge] (0.20000, -0.10000, -1.00000) -- (0.60000, -0.10000, -1.00000);
\draw[edge] (0.60000, -0.10000, -1.00000) -- (1.00000, -0.60000, 0.10000);
\draw[edge] (0.60000, -0.10000, -1.00000) -- (1.00000, -0.20000, 0.10000);
\draw[edge] (1.00000, -0.60000, 0.10000) -- (1.00000, -0.20000, 0.10000);
%%
%%
%% Drawing the vertices in the front
%%
\node[vertex] at (-0.60000, 0.10000, 1.00000)     {};
\node[vertex] at (-0.20000, 0.10000, 1.00000)     {};
\node[vertex] at (-0.10000, -1.00000, 0.20000)     {};
\node[vertex] at (-0.10000, -1.00000, 0.60000)     {};
\node[vertex] at (0.10000, 1.00000, -0.60000)     {};
\node[vertex] at (0.10000, 1.00000, -0.20000)     {};
\node[vertex] at (0.20000, -0.10000, -1.00000)     {};
\node[vertex] at (0.60000, -0.10000, -1.00000)     {};
\node[vertex] at (1.00000, -0.60000, 0.10000)     {};
\node[vertex] at (1.00000, -0.20000, 0.10000)     {};
\end{tikzpicture}}
  \end{center}
  \caption{The unit ball from \Cref{exam:icosahedron}.}
    \label{fig:icosahedron}
  \end{figure}
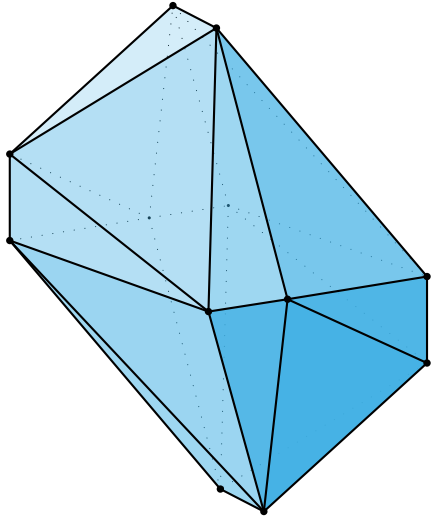
\end{exam}

Despite \Cref{thm:lower-bound-upper-left} not all hope is lost.
We can still obtain bounds on the joint spectral radius of submatrices.
Our first result in this direction produces a transformation
which make the joint spectral radius of all submatrices small simultaneously. 

\begin{thm}
  \label{thm:all-submatrices}
    Let $\calM \subseteq \K^{d \times d}$ be compact.
    Then there exists $T \in \gl_d(\K)$ such that
    $\rho((T^{-1}\calM T)_{J,J}) \leq \sqrt{d} \rho(\calM)$
    for every $J \subseteq \{1,\dots,d\}$.
  \end{thm}
  \begin{proof}
    First notice that the case $\calM = \{0\}$ is trivial.
    Now assume that $\calM$ is irreducible.
    Choose an extremal norm $\norm{\cdot}$
    with unit ball $K$. There is $T \in \gl_d(\K)$
    such that $T^{-1}K$ is in John's position, i.e.
    the Euclidean unit ball $B_{1,{\norm{\cdot}_2}}$
    is the ellipsoid of maximal volume inscribed in $K$.
    % x \in T^{-1}K <-> \norm{Tx}=1 
    By John's theorem $T^{-1}K$ is then contained
    in $B_{\sqrt{d},{\norm{\cdot}_2}}$.
    This theorem is usually only stated over the reals
    but it also holds for $\K=\C$, see
    \cite[Section 8, p. 46]{johnsonBasicConceptsGeometry2001}.
    In terms of the norm this means
    $\frac{1}{\sqrt{d}} \norm{x}_2 \leq \norm{Tx} \leq \norm{x}_2$.
    Let $J \subseteq \{1,\dots,d\}$ be an index set of cardinality $m$.
    Let $z \in \K^m$.
    Let $\nnorm{\cdot}$ be the norm on $\K^m$ defined
    by $\nnorm{z} := \norm{T I_J z}$.
    
    For $A \in \calM$ and $z \in \K^m$ we get
    \begin{align*}
      \nnorm{(T^{-1}AT)_{J,J}z}
            &= \nnorm{I_J^\top T^{-1} AT I_J z} \\
            &= \norm{T P_J T^{-1} A T I_J z} \\
            &\leq \norm{P_J T^{-1} A T I_J z}_2 \\
            &\leq \norm{P_J}_2\norm{T^{-1} A T I_J z}_2 \\
            &\leq \sqrt{d} \norm{A} \norm{T I_J z} \\
            &\leq \sqrt{d} \rho(\calM)\nnorm{z}.
    \end{align*}
    If we take the supremum over $z \in \mathbb{K}^m$,
    we get \[\rho((T^{-1} A T)_{J,J}) \leq \nnorm{(T^{-1}AT)_{J,J}} \leq \sqrt{d}
      \rho(\calM).\]
    If $\calM$ is reducible, there is $S \in \gl_d(\K)$
    such that $S^{-1} \calM S$ is in upper block triangular
    form and all diagonal blocks are irreducible.
    For each of these diagonal blocks $B_i$ we can
    find an invertible matrix $R_i$ such that $\rho((R_i^{-1} B_i
    R_i)_{K,K}) \leq \sqrt{d} \rho(\calM)$
    for all suitable index sets $K$.
    If $R$ is the block diagonal matrix with
    diagonal blocks $R_i$, then $T:=SR$
    satisfies the conclusion in our theorem.
  \end{proof}

  If we are only interested in bounds on the JSR of a single submatrix
  family of size $m$, we can improve the bound from
  $\sqrt{d} \varrho(\calM)$ to $\sqrt{m} \varrho(\calM)$, with a slight further enhancement beyond this.
  \begin{thm}
    \label{thm:one-submatrix}
    Set 
    \begin{align*}
        \delta_{\K}(m):= \begin{cases}
            \frac{2}{m+1}(1+\frac{m-1}{2}\sqrt{m+2}) &\text{ for } \K=\R,\\
            \frac{1}{m}(1+(m-1)\sqrt{m+1}) &\text{ for } \K=\C.
        \end{cases}
    \end{align*}
    Let $\calM \subseteq \K^{d \times d}$ be compact
    and let $J \subseteq \{1,\dots,d\}$
    be an index set of cardinality $m$.
    Then there exists $T \in \gl_d(\K)$ such that
    \[\rho((T^{-1}\calM T)_{J,J}) \leq \delta_\K(m) \rho(\calM) \leq \sqrt{m}\rho(\calM).\]
  \end{thm}
  \begin{proof}
    If $\calM=\{0\}$, the result is trivial.
    We first treat the irreducible case.
    By scaling we can assume that $\rho(\calM)=1$.
    Let $\norm{\cdot}$ be an extremal norm
    and let $J \subseteq \{1,\dots,d\}$ be an index set
    of cardinality $m$.
    Let $V$ be the image of $I_J$.
    By \cite[Theorem 2.3]{deregowskaSimpleProofGrunbaum2023}
    there is a projection $Q$ on $\K^d$
    with image $V$ and operator norm $\norm{Q} \leq \delta_\K(m)$.
    Let $w_i, i \not \in J$ be a basis of the kernel of $Q$.
    Let $T$ be the matrix obtained
    from the $d \times d$ identity matrix
    by replacing the $i$-th column with $w_i$
    for $i \not\in J$.
    Then
    $Te_j=e_j$ for $j \in J$ and $Te_i=w_i$
    for $i \not \in J$. Hence $TI_J=I_J$ and $P_J=QT$.
    Define a norm $\nnorm{\cdot}$ on $\K^m$
    by $\nnorm{z}:=\norm{I_J z}$. For $z \in \K^m$ and $A \in \calM$
    we get
    \begin{align*}
      \nnorm{(T^{-1}AT)_{J,J}z}
            &= \nnorm{I_J^\top T^{-1}AT I_J z} \\
            &= \norm{P_J T^{-1} A T I_J z} \\
            &= \norm{Q A I_J z} \\
            &\leq \norm{Q} \norm{A} \nnorm{z}\\
            &\leq \delta_\K(m) \rho(\calM) \nnorm{z}.
    \end{align*}
    Taking the supremum over $z \in \mathbb{K}^m$,
    we get \[\rho((T^{-1} A T)_{J,J}) \leq \nnorm{(T^{-1}AT)_{J,J}} \leq \delta_\K(m) \rho(\calM).\]

    For the reducible case notice first, that $\delta_\K(m)$ is 
    increasing in $m$.
    If $\calM$ is reducible, there is $S$
    such that $S^{-1} \calM S$ is in upper block triangular
    form and all diagonal blocks are irreducible.
    Let $J_i$ be the intersection of $J$ with
    the index set of the $i$-th diagonal block.
    For each of the diagonal blocks $B_i$ we can
    find an invertible matrix $R_i$ such that 
    \[\rho((R_i^{-1} B_i R_i)_{J_i,J_i}) \leq 
    \delta_\K(\abs{J_i}) \rho(\calM) \leq \delta_\K(m) \rho(\calM).\]
    
    If $R$ is the block diagonal matrix with
    diagonal blocks $R_i$, then $T:=SR$
    satisfies the conclusion in our theorem.
  \end{proof}

\section{Pairs of matrices}
\label{sec:pairs}

For pairs of real matrices, alternative normalizations may be considered by using results on hollowization obtained by Damm and Fa{\ss}bender in \cite{damm2020simultaneous}. In particular, it is possible to equalize all diagonal entries of a pair of matrices - up to two entries.
We will say that $A\in\R^{d\times d}$ has {\em zero diagonal}, if $a_{ii}=0$, $i=1,\ldots,d$. In \cite{damm2020simultaneous} this is called \emph{hollow}. A matrix $A\in \R^{d\times d}$ with zero trace is called \emph{almost hollow}, if  $a_{ii}=0$, $1=1,\ldots,d-2$.
We recall the result for convenience.

\begin{thm}[\cite{damm2020simultaneous}]
\label{thm:dammfasssim}
    Let $A,B \in \R^{d\times d}$ with $\trace A = \trace B = 0$. Then there exists an orthogonal matrix $T$ such that $T^{-1}AT$ is hollow and $T^{-1}BT$ is almost hollow. 
\end{thm}

It is shown in \cite[Remark~16]{damm2020simultaneous} that a similar statement does not hold for general pairs of complex matrices and unitary transformations.

\begin{thm}
\label{thm:hollow}
    For every pair $\calM= \{A,B\} \subseteq \R^{d\times d}$, $\rho(\calM)>0$, there exists a similarity transformation $T\in \gl_d(\R)$ such that the Euclidean norm of all columns and rows of $(T^{-1}AT,T^{-1}BT)$ is bounded by $\sqrt{d}\rho(\calM)$ and so that
    \begin{equation}
        T^{-1}AT- \frac{\trace A}{d}I_d  \text{ has zero diagonal and }
        T^{-1}BT- \frac{\trace B}{d}I_d \text{ is almost hollow}.
    \end{equation}
\end{thm}

\begin{proof}
    First apply \Cref{prop:diagonal1} to arrive at matrices with the desired bound on the Euclidean norm of columns and rows. Then apply Theorem~\ref{thm:dammfasssim}.
\end{proof}

We note that $|\trace A|, |\trace B| \leq d\rho(\calM)$ in the situation of \Cref{thm:hollow}.

\section{Conclusion}
\label{sec:outlook}

For a given compact irreducible set of matrices $\calM$ we used an Auerbach basis
of an associated extremal norm to construct
a similarity transformation that renders all the entries in the transformed matrices smaller
or equal to the joint spectral radius of $\calM$.

The procedure to bound the entries of
the transformed matrices is not optimal in the following sense.
Already the spectral radius of a single matrix might be larger than
the maximal absolute value of any of its entries.
This happens e.g. if all entries of the matrix equal one.
Hence, representing matrix sets
using an Auerbach basis of the extremal norm might actually increase
the entries of the matrices.
It might be interesting to find, for a given compact
set of matrices, a similarity transformation 
that minimizes the maximal entry of all transformed matrices. 
In a larger context, we
can ask this question for other norms besides the supremum
norm: If we define the norm of a matrix set
to be the maximum of the norms of its elements, what
can we say about the minimum of the norms of the matrix sets
in a similarity orbit?
How are these quantities related for different norms?
For an example of such a result,
see \cite[Lemma 1]{breuillard2022joint}.

For higher dimensional principal submatrices we showed
that the JSR 
of the submatrix set can in general only be bounded up to a multiplicative 
factor by the JSR of the original matrix set. The bounds for these factors presented in \Cref{sec:higher-dimensionalPS} are certainly not optimal.
However, by \Cref{rem:shadinessconstant} the problem of finding such bounds 
  is a problem purely about the geometry of finite dimensional Banach spaces.

\bibliographystyle{abbrvurl} 
\bibliography{bibliography}
\end{document}